\documentclass[12pt]{amsart}
\usepackage{amsmath,amssymb,enumerate}

\usepackage[all]{xy}

\newcommand{\Br}{{\mathrm{Br}}}

\newcommand{\Der}{{\mathrm{Der}}}

\newcommand{\rmc}{{\mathrm{c}}}

\newcommand{\rms}{{\mathrm{s}}}
\newcommand{\rmII}{\mathrm{II}}

\newcommand{\rAd}{\mathrm{Ad}}

\newcommand{\bbC}{{\mathbb{C}}}

\newcommand{\bbH}{{\mathbb H}}

\newcommand{\bbO}{{\mathbb O}}
\newcommand{\bbR}{{\mathbb R}}

\newcommand{\bbZ}{{\mathbb Z}}

\newcommand{\ok}{{\overline{k}}}

\newcommand{\rA}{{\mathrm{A}}}

\newcommand{\rE}{{\mathrm{E}}}
\newcommand{\rF}{{\mathrm{F}}}
\newcommand{\rG}{\mathrm{G}}

\newcommand{\rM}{{\mathrm{M}}}

\newcommand{\SL}{{\mathrm{SL}}}

\newcommand{\GL}{{\mathrm{GL}}}

\newcommand{\frake}{{\mathfrak{e}}}
\newcommand{\frakf}{{\mathfrak{f}}}
\newcommand{\frakg}{{\mathfrak{g}}}
\newcommand{\frakh}{{\mathfrak{h}}}

\newcommand{\frakk}{{\mathfrak k}}

\newcommand{\fraksl}{{\mathfrak{sl}}}
\newcommand{\frakso}{{\mathfrak{so}}}
\newcommand{\fraksp}{{\mathfrak{sp}}}
\newcommand{\fraksu}{{\mathfrak{su}}}

\newcommand{\frakp}{{\mathfrak p}}

\newcommand{\Aut}{{\mathrm{Aut}}}

\newtheorem{lemma}[subsection]{Lemma}

\newtheorem{prop}[subsection]{Proposition}
\newtheorem{thm}[subsection]{Theorem}

\newcommand{\calA}{{\mathcal{A}}}

\newcommand{\calC}{{\mathcal{C}}}

\newcommand{\calH}{{\mathcal{H}}}

\newcommand{\calJ}{{\mathcal{J}}}

\begin{document}


\title[Exceptional dual pairs]{rational forms of exceptional dual pairs}

\author{Hung Yean Loke}

\author{Gordan Savin}

\address{Hung Yean Loke, Department of Mathematics,
National University of Singapore,
   21 Lower Kent Ridge Road,
   Singapore 119077.}
\email{matlhy@nus.edu.sg}
\address{Gordan Savin, Department of
  Mathematics, University of Utah, Salt Lake City, UT 84112}
\email{savin@math.utah.edu}

\begin{abstract}
Let $\frakg$  be a simple Lie algebra of type $\rF_4$ or $\rE_n$ defined over a local or global field $k$ of characteristic zero. 
We show that $\frakg$ can be obtained by the Tits construction from an octonion algebra $\bbO$ and a cubic Jordan algebra $\calJ$. 
In particular,  $\frakg$ contains a dual pair $\frakh$ defined over $k$ which is the direct sum of the derivation algebras of $\bbO$ and $\calJ$. 
We determine the conjugacy classes of $k$-forms of $\frakh$ in $\frakg$.

\end{abstract}


\maketitle

\section{Introduction} \label{S1} 

\subsection{}
Let $k$ be a field of characteristic 0. Let $\bbO$ be an octonion $k$-algebra and  let $\Der(\bbO)$ be the Lie algebra of derivatives. 
It is a simple Lie algebra of exceptional type $\rG_2$.  
Let $\calJ$ be a central simple cubic Jordan $k$-algebra and let  $\Der(\calJ)$ be the Lie algebra of derivatives. 
Let $\bbO_0$ be the set of trace 0 elements in $\bbO$, and let $\calJ_0$ be the set of trace 0 elements in $\calJ$. 
Then $\frakh=\Der(\bbO)\oplus\Der(\calJ)$ acts on $\frakh^{\perp}=\bbO_0\otimes\calJ_0$.  It is an absolutely irreducible 
representation of $\frakh$. On the vector space 
\[
\frakg=\frakg(\bbO,\calJ)=\frakh\oplus \frakh^{\perp} 
\] 
Tits \cite{T1} defines a Lie structure such that $\frakh$ is a subalgebra, and the bracket of $x\in\frakh$ and $y\in\frakh^{\perp}$ 
is the action of $\frakh$ on $\frakh^{\perp}$.  The Lie algebra $\frakg$ is simple, and of  
type $\rF_4$ or $\rE_n$, depending on the Jordan algebra.  The centralizer of $\Der(\bbO)$ in $\frakg$ is $\Der(\calJ)$, and 
vice versa. In other words, the two simple factors of $\frakh$ form a dual pair in $\frakg$. 
 
The algebra $\frakh$ is a maximal subalgebra in $\frakg$. If $k$ is an algebraically closed field, Dynkin \cite{Dy} showed that any subalgebra 
of $\frakg$ isomorphic to $\frakh$ is conjugated to $\frakh$ by an automorphism of $\frakg$.  Let $A=\Aut(\frakg)$ be the group of 
automorphisms of $\frakg$. 
Let $X(\bar k)$ denote the set of all
$A(\bar k)$-conjugates of $\frakh \otimes \bar k$.  Then 
$X(\ok) = A(\ok)/B(\ok)$ where $B$ is the normalizer of $\frakh$ in $A$.  It  is an algebraic variety
defined over $k$ (see \cite{B}). We have an exact sequence of pointed sets 
\[
1 \rightarrow B \rightarrow A \rightarrow X \rightarrow 1
\]
where $g \in A(\ok)$ maps to $\rAd_g(\frakh \otimes \ok)$ in
$X(\ok)$, and the corresponding long exact sequence of pointed sets
\[
\xymatrix{1 \ar[r] & B(k) \ar[r] & A(k) \ar[r] & X(k) \ar[r] &
  H^1(k,B) \ar[r]^\gamma & H^1(k,A)}
\] 
The group $B$ acts on $\frakh$ by conjugation i.e. we have a canonical map from $B$ to $\Aut(\frakh)$. This map is an 
isomorphism. It follows that the sets $H^1(k,B)$ and $H^1(k,A)$ classify forms of $\frakh$ and $\frakg$, respectively. It is 
not too difficult to see that the map $\gamma$ is the Tits construction.  Thus, in order to describe $B(k)$-orbits on $X(k)$, i.e. rational 
conjugacy classes of dual pairs, it suffices to determine the fibers of the Tits construction.
If $k$ is a local field, the map $\gamma$  can be completely described. It is always a surjection. In fact, it is a bijection 
unless $k\cong\bbR$ or $\frakg$ is of type $\rF_4$. 
For a global field $k$ we have the following.

\begin{thm} \label{main}
Let $k$ be a number field. Let $S$ be the set of archimedean places, i.e.  $k_{\infty}=k\otimes\bbR=\prod_{v\in S} k_v$. 
The map $\gamma: H^1(k,B)\rightarrow H^1(k,A)$ is surjective. If 
$A$ is not of type $\rF_4$ then the fibers of $\gamma$ are finite. More precisely, if the fiber of $[\frakg]\in  H^1(k,A)$
consists of  $n$ distinct classes $[\frakh_1], \ldots, [\frakh_n]\in H^1(k,B)$, then the fiber of $[\frakg\otimes\bbR]\in H^1(k_{\infty},A)=\prod_{v\in S} H^1(k_v, A)$
consists of $n$ distinct  classes  $[\frakh_1\otimes\bbR], \ldots, [\frakh_n\otimes\bbR]\in H^1(k_{\infty},B)=\prod_{v\in S} H^1(k_v, B)$. 
 For any finite place $v$ of $k$,  $[\frakh_1\otimes k_v]=\ldots =[\frakh_n\otimes k_v]\in H^1(k_v,B)$.  
\end{thm} 
\noindent The proof of this theorem is an application of the Hasse principle, keeping in mind that  the groups $B$ and $A$ are not necessarily connected, which presents additional difficulties.

\subsection*{Acknowledgment}
The first author enjoyed hospitality of the University of Utah  and Tata Institute in Mumbai while writing this paper. 
The second author would like to thank Skip Garibaldi for enlightening  discussions, and NSF for continued support.

\section{Preliminaries}

\subsection{Galois cohomology}  \label{S2.1} 
Let $\frakg$ be a semi-simple Lie algebra over $k$.  Let $G$  and $A$ be the groups of inner 
and outer automorphisms of $\frakg$, respectively.  The Galois cohomology  $H^1(k, A)$ classifies $k$-forms of $\frakg$. 
The forms of $\frakg$ whose cohomology classes are in the image of the map $j: H^1(k,G)\longrightarrow H^1(k,A)$ 
are called inner forms. The map $j$ is not injective in general, so the elements in $H^1(k,G)$ classify inner forms of 
$\frakg$ with some additional information. 
Assume that 
we have a short exact sequence 
\[ 
\xymatrix{1\ar[r] &  G(\bar k) \ar[r] & A(\bar k) \ar[r] & \bbZ/2 \bbZ \ar[r]  &1 &} 
\] 
and, hence,  a long exact sequence  
\[
\xymatrix{1 \ar[r] & G(k) \ar[r] & A(k) \ar[r] & H^0(k, \bbZ/ 2
  \bbZ)\ar[r] & H^1(k,G) \ar[r]^j& H^1(k,A )}
\] 
of pointed sets in Galois cohomology.  The Galois group $\Gamma_k$
must act trivially on $\bbZ/2\bbZ$ so $H^0(k,\bbZ/2\bbZ) =
\bbZ/2\bbZ$.

 Let $[\frakg'] \in H^1(k,A)$ be in the image of $j$. By the standard twisting argument, 
the exact sequence shows that $j^{-1}([\frakg'])$ consist of one or two points. It consist of one point if and only if 
$A'(k)/G'(k)$ is nontrivial. i.e. $\frakg'$ has an outer  automorphism defined over $k$. The image of $j$ contains a unique 
quasi split form $\frakg^{qs}$.  It is well known that $\frakg^{qs}$ has an outer automorphism defined over $k$. Hence $H^1(k,G)$ 
contains a unique point $c_G$ such that $j(c_G)=[\frakg^{qs}]$.

\subsection{Lie algebras and groups of type $\rG_2$} \label{S25}
Let $\bbO$ be an octonion $k$-algebra. The space of derivatives $\frakg_2 = \Der(\bbO)$ is a simple Lie $k$-algebra of type
$\rG_2$, and $G = \Aut(\frakg_2)=\Aut(\bbO)$ is a connected, simply connected algebraic $k$-group of type
$\rG_2$. Hence $H^1(k,G)$ classifies any of the following the collections of objects (see page 190 in \cite{Se}):
\begin{enumerate}[(a)]
\item Octonion $k$-algebras.
 
\item Lie $k$-algebras of type $\rG_2$. 
 
\item Connected algebraic groups of type $\rG_2$ over $k$.
\end{enumerate}

\subsection{Jordan algebras} \label{S26} Let $\calA$ be a simple
$k$-algebra.  Let $\calA^+$ be the Jordan algebra with the
same 
space and multiplication $x\cdot y = \frac{1}{2}(xy + yx)$.  An
involution of $\calA$ is a linear map $\sigma : \calA \rightarrow
\calA$ such that $\sigma(xy)=\sigma(y)\sigma(x)$ for all $x,y\in
\calA$.  Then $S(\sigma,\calA)=\{x\in \calA | \sigma(x)=x\}$ is a
Jordan subalgebra of $\calA^+$.  Let $\calJ$ be a central simple cubic
Jordan algebra. Then $\calJ$ is one of the following five types (see
Theorem 37.2(2) and Section 37.C in \cite{BoI}):
\begin{enumerate} [(a)]
\item $\calJ=S(\sigma,\calA)$ where $\calA$ is a central simple
  $k$-algebra of degree 3, and $\sigma$ is an orthogonal involution.
  Then $\dim(\calJ)=6$, and the Lie algebra $\Der(\calJ)$ is a form of
  $\frakso(3)$.

\item $\calJ=\calA^+$ where $\calA$ is a central simple $k$-algebra of
  degree 3.  Then $\dim(\calJ)=9$, and the Lie algebra $\Der(\calJ)$
  is an inner form of $\fraksl(3)$.

\item $\calJ=S(\sigma, \calA)$ where $\calA$ is a central simple
  $K$-algebra of degree 3, $K$ is a quadratic extension of $k$, and
  $\sigma$ is a unitary involution. Then $\dim(\calJ)=9$, and the Lie
  algebra $\Der(\calJ)$ is an inner form of $\fraksu(3)$.

\item $\calJ=S(\sigma,\calA)$ where $\calA$ is a central simple
  $k$-algebra of degree 6, and $\sigma$ is a symplectic involution.
  Then $\dim(\calJ)=15$, and the Lie algebra $\Der(\calJ)$ is a form
  of $\fraksp(6)$.

\item $\calJ$ is an exceptional Jordan algebra. Then $\dim(\calJ)=27$,
  and the Lie algebra $\Der(\calJ)$ is a form of $\rF_4$.
\end{enumerate}

A convenient way to construct (some) cubic Jordan algebras is by means
of composition algebras.  Let $\calC$ be a composition
$k$-algebra. Let $c\mapsto \bar c$, for $c\in \calC$, be the
conjugation in $\calC$.  Let
$\varrho=diag(\varrho_1,\varrho_2,\varrho_3)$ be a 3 by 3 diagonal
matrix with entries in $k$. We set
\[
\calH_3(\calC,\varrho) = \{ x \in \rM_3(\calC) :
\bar{x}^\top = \varrho x \varrho^{-1} \}.
\]
This forms a Jordan algebra under the multiplication defined by $x
\cdot y = \frac{1}{2}(xy + yx)$.  See Section 37.C in \cite{BoI}. If
$\varrho$ is the identity matrix, then we denote the above Jordan
algebra by $\calH_3(\calC)$.  The dimension of
$\calH_3(\calC,\varrho)$ is $3+3\dim(\calC)$.  If $k=\bbR$ then we say
that a composition algebra $\calC$ is definite if $c\mapsto c\bar c$
is a positive definite quadratic form. Otherwise, we say that $\calC$
is indefinite. Definite composition algebras are $\bbR,\bbC,\bbH$ and
Cayley octonions $\bbO$.  The following is an exhaustive list of all
real central simple cubic Jordan algebras (see pages 172-173 in
\cite{OV}):

\begin{enumerate}[(a)]
\item $\calJ^{c}= \calH_3(\calC,\varrho)$ where $\calC$ is definite,
  and $\varrho={\mathrm{diag}}(1,1,1)$.  The real Lie algebra
  $\Der(\calJ^c)$ is compact.

\item $\calJ^{\rmII}= \calH_3(\calC,\varrho)$ where $\calC$ is
  definite, and $\varrho={\mathrm{diag}}(1,-1,1)$.  The real Lie
  algebra $\Der(\calJ^{\rmII})$ has split rank one.

\item $\calJ^{s}= \calH_3(\calC,\varrho))$ where $\calC$ is
  indefinite, and $\varrho={\mathrm{diag}}(1,1,1)$.  The real Lie
  algebra $\Der(\calJ^s)$ is split.
\end{enumerate}

\section{Tits construction} \label{S28} 
\subsection{} Let $\bbO$ be an octonion $k$-algebra and $\calJ$ a
central simple cubic Jordan $k$-algebra.  Let $\Der(\bbO)$ and
$\Der(\calJ)$ be the Lie algebras of derivatives.  Let
$\frakh=\Der(\bbO)\oplus \Der(\calJ)$.  Let $H$ be the group of inner
automorphisms of $\frakh$.  Let $\bbO_0$ and $\calJ_0$ denote the
respective subspaces of trace zero elements. Then
$\frakh^{\perp}=\bbO_0\otimes \calJ_0$ is an irreducible $\frakh$
module. It is also an irreducible $H$-module.  In \cite{T1}, Tits
constructed an exceptional Lie algebra $\frakg$ over $k$ on the vector
space
\[ 
\frakg = \frakg(\bbO, \calJ) = \frakh  \oplus \frakh^{\perp} 
\]
We refer to \cite{T1} and \cite{Ja} for the definition of the Lie
brackets.  If $\dim \calJ = 6, 9, 15$ or $27$, then~$\frakg$ is a
simple Lie algebra of type $\rF_4, \rE_6, \rE_7$ or $\rE_8$
respectively. Henceforth, we shall call $(\frakh,\frakg)$ a Tits pair. 

Let $G$ be the group of inner automorphisms of $\frakg$.  As $H$ acts
on $\frakh^{\perp}$, it is immediate that $H$ is contained in $G$.
Let $G_{sc}$ and $H_{sc}$ be simply connected covers of $G$ and $H$,
respectively.  Let $C_G$ and $C_H$ be their centers.  We have the
following important commutative diagram with exact rows that will be
used throughout the paper.
\begin{equation} \label{BIG}
\xymatrix{ 1 \ar[r] & C_H \ar[r] \ar@{->}[d] & H_{sc} 
   \ar[r] \ar@{->}[d] & H
  \ar[r] \ar@{->}[d] & 1 \\ 1 \ar[r] & C_G \ar[r] & G_{sc}
  \ar[r] & G \ar[r] & 1}
\end{equation}
If $G$ is of type $\rE_6$ (resp. $\rE_7$, $\rE_8$), then $H_{sc}$ is a
subgroup of $G_{sc}$ and $C_H=C_G$ is a cyclic group of order 3
(resp. 2, 1).  If $G$ is of type $\rF_4$, then $C_G$ is trivial but
$C_H$ is a cyclic group of order~2.

\begin{lemma} \label{L29} 
Let $(\frakh,\frakg)$ be a Tits pair defined
  over $k$. Let $B$ be the normalizer of $\frakh$ in $A=\Aut(\frakg)$.

\begin{enumerate}[(i)]
\item The centralizer of $\frakh$ in $A(k)$ is trivial.

\item The conjugation action of $B$ on $\frakh$ gives an isomorphism
  of $B(k)$ and $\Aut(\frakh)(k)$.
\end{enumerate}
\end{lemma}

\begin{proof}
  Let $\varphi\in A(k)$ such that $\varphi(x)=x$ for all $x\in\frakh$.
  Since $\frakh^{\perp}$ is an absolutely irreducible $\frakh$-module,
  there exists a non-zero constant $c$ such that $\varphi(x)=cx$ for
  all $x\in\frakh^{\perp}$. From the explicit formula for the Lie
  bracket in \cite{T1}, it follows that there exist $x,y\in
  \frakh^{\perp}$ such that $[x,y]$ is not contained in $\frakh$ i.e.
  $[x,y]=z+z^{\perp}\in \frakh\oplus \frakh^{\perp}$ and
  $z^{\perp}\neq0$. Applying $\varphi$ to this relation implies
  $c^2=c$. Hence $c=1$ and $\varphi$ is the trivial automorphism. This
  proves~(i).  

  Part (i) implies that the conjugation action of the normalizer $B$
  on $\frakh$ is faithful. Since $H(k)\subseteq B(k) \subseteq
  \Aut(\frakh)(k)$, we are done, except in the $\rE_6$-case, when
  $H(k) \neq \Aut(\frakh)(k)$ i.e. when $\frakh$ has an outer
  automorphism defined over $k$. Recall that a form of $\fraksl(3)$
  corresponds to a central simple algebra $ \calA$ of degree 3 with a
  unitary involution $\sigma$ over an \'{e}tale quadratic extension
  $K$ of $k$.  If the form of $\fraksl(3)$ has an outer automorphism
  defined over $k$, then~$\calA$ is isomorphic to its opposite. Hence
  $\calA$ is isomorphic to the matrix algebra $M_3(K)$. Then $\calJ=\{
  x\in M_3(K) ~|~ \sigma(x)=x\}$, and $\Der(\calJ)=\{ x\in M_3(K) ~|~
  \sigma(x)=-x\}$. For every $x\in M_3(K)$ let $\bar{x}$ denote the
  action of the non-trivial automorphism of $K$ over $k$ .  Then
  $\tau(x,y)=(x, \bar{y})$, for every $(x,y)\in \Der(\mathbb O) \oplus
  \Der(\calJ)$, is an outer automorphism of $\frakh$.  It extends to
  an outer automorphism of $\frakg=\frakh \oplus \frakh^{\perp}$ by
  $\tau(x\otimes y)=x\otimes \bar{y}$, for every $x\otimes y\in
  \mathbb O_0 \otimes \calJ_0$.  Part (ii) follows.
 \end{proof}

 \subsection{} \label{S22} Let $k$ be any field, and $\frakh\subseteq
 \frakg$ be two semi-simple Lie algebras defined over~$k$.  Let $A =
 \Aut(\frakg)$ and $B$ the normalizer of $\frakh$ in $A$. We have a
 natural map from $B$ to $\Aut(\frakh)$. Assume that this is a
 bijection, as it is in the case of a Tits pair $(\frakh,\frakg)$.
 Let $\frakh'\subseteq\frakg'$ be a rational $k$-form of $\frakh
 \subseteq \frakg$, i.e. there exists an isomorphism $\varphi:
 \frakg(\bar k) \rightarrow \frakg'(\bar k)$ such that $\varphi
 (\frakh(\bar k))=\frakh'(\bar k)$.
 
\begin{lemma} \label{L:unique} 
The class of $\frakh'$ determines the class of $\frakg'$. In particular, if $\frakh$
and $\frakh'$ are isomorphic over $k$, then $\frakg$ and $\frakg'$ are isomorphic
over $k$ too.
 \end{lemma} 

\begin{proof}
 The class $[\frakg']\in H^1(k,A)$ is given by the cocycle
 \[ 
 {\sigma}\mapsto a_{\sigma}= \varphi^{-1} \circ \varphi^{\sigma} \in
 A(\bar{k})
 \] 
 where $\sigma\in\Gamma_k$.  Since $\varphi^{-1} \circ
 \varphi^{\sigma}(\frakh) = \frakh$, it follows that $a_{\sigma}\in
 B(\bar k)$.  Hence $[\frakg']\in H^{1}(k,B)$.  Under the isomorphism
 of $B$ and $\Aut(\frakh)$, the class $[\frakg']$ corresponds to the
 cohomology class $[\frakh']$. Hence the lemma.
\end{proof}

\begin{prop} Let $(\frakh,\frakg)$ be a Tits pair. The map 
\[ 
\gamma: H^1(k,B)\longrightarrow H^1(k,A) 
\] 
is the Tits construction. 
\end{prop} 
\begin{proof} Indeed, if $\frakh'$ is a form of $\frakh$, then $\gamma([\frakh'])=[\frakg']$ where 
$\frakg'$ is a form of $\frakg$ containing $\frakh'$. We know, from Lemma \ref{L:unique}, that there is at most one form of 
$\frakg$ containing $\frakh'$. But one such form is obtained by the Tits construction. Hence $\frakg'$ is obtained 
from $\frakh'$ by the Tits construction. 
\end{proof} 

\begin{prop}\label{P35} 
 Let $(\frakh,\frakg)$ be a Tits pair. If $\frakh$ is split (respectively quasi-split) then $\frakg$ is split (respectively quasi-split). 
\end{prop} 
\begin{proof} By Lemma \ref{L:unique}, it suffices to construct, by
  any means, a rational $k$-form $\frakh' \subseteq \frakg'$ of
  $\frakh \subseteq \frakg$ where both $\frakh'$ and $\frakg'$ are
  (quasi) split.  Let $\frakg'$ be a split Lie algebra over $k$. It is
  well known that there is a split dual pair $\frakh'$ in $\frakg'$ see,
  for example, \cite{Sa} for a construction. We shall extend it to the
  quasi-split case. We may assume that $\frakg'$ has type $\rE_6$.  Lift
  the nontrivial automorphism of the Dynkin diagram to an outer
  automorphism $\tau$ of $\frakg'$.  Extend the Dynkin diagram by
  adding the lowest root, and remove the branch point. This is a
  diagram of type $\rA_2\times\rA_2\times\rA_2$. Hence $\frakg'$
  contains a subalgebra isomorphic to $\fraksl_3(k) \oplus
  \fraksl_3(k) \oplus \fraksl_3(k)$. One can choose an identification
  such that
\begin{enumerate}[(a)]
\item $\tau (x,y,z)=(x,z,y)$ for all $(x,y,z)\in \fraksl_3(k) \oplus
  \fraksl_3(k) \oplus \fraksl_3(k)$ and, 

\item under the adjoint action of
  $\fraksl_3(k) \oplus \fraksl_3(k) \oplus \fraksl_3(k)$, the Lie
  algebra $\frakg'$ decomposes as
\[ 
\frakg' = \fraksl_3(k) \oplus \fraksl_3(k) \oplus \fraksl_3(k) \oplus
(V\otimes V\otimes V )\oplus (V^*\otimes V^* \otimes V^*)
\] 
where $V$ is the standard representation of $\fraksl_3(k)$. 
\end{enumerate}
Lift the nontrivial automorphism of the Dynkin diagram of type $\rA_2$
to an outer automorphism $\rho$ of $\fraksl_3(k)$. Let $\varphi :
\fraksl_3(k) \rightarrow \frakg'$ be a homomorphism defined by
\[
\varphi(x)=(0,x,\rho(x)) \in \fraksl_3(k) \oplus \fraksl_3(k) \oplus
\fraksl_3(k)\subseteq \frakg'
\] 
for every $x\in \fraksl_3(k)$.  The centralizer of $\varphi(\fraksl_3(k))$ in $\frakg'$ is
\[ 
 \frakg_2\cong \fraksl_3(k) \oplus V \oplus V^*
\] 
a split Lie algebra of type $\rG_2$. Thus $\frakg'$ contains a split
dual pair $\frakh' \cong \frakg_2\oplus \fraksl_3(k)$. Moreover, the
automorphism $\tau$ of $\frakg'$ restricts to an automorphism of
$\frakh'$, trivial on $\frakg_2$, and $\rho$ on $\fraksl_3(k)$.  In
particular, if $K$ is a quadratic extension of $k$, we can twist
simultaneously $\frakh'$ and $\frakg'$ to construct a quasi split form
of $\frakh'$ contained in a quasi split form of $\frakg'$.
\end{proof}

\section{Tits algebras}

In this section we will review representations of algebraic groups over $k$ from \cite{T2}.

\subsection{}
Let $G_{sc}$ be a semi-simple, simply connected  algebraic group defined over $k$.  
Let $T$ be a maximal torus of $G_{sc}$ defined over $k$. Let $\Lambda$
denote the lattice of weights of $T(\bar k)$. Let $\Lambda_r$ denote the
root sublattice.  We choose a positive root system $\Phi^+$ in
$\Lambda$.  Let $\Lambda_+$ denote the subset of dominant weights. We
set $C^* = \Lambda / \Lambda_r$ which is the abelian group dual to the
center $C(\ok)$ of~$G_{sc}(\ok)$.

The Galois group $\Gamma_k$ acts on the Dynkin diagram as follows. 
If $\sigma \in \Gamma_k$ then the conjugate of a weight $\lambda$ is 
defined by ${}^{\sigma}\lambda=\sigma\circ \lambda\circ\sigma^{-1} $.  For every $\sigma\in \Gamma_k$ 
there exists a unique element $w$ in the Weyl group such that $w({}^\sigma
\Phi^+) = \Phi^+$.  In particular, $\lambda\mapsto w({}^{\sigma}\lambda)$ defines a 
permutation of simple roots. For example, let $K$ be a quadratic extension of $k$ and consider a hermitian form 
\[
h(z_1,\ldots,z_n)=a_1|z_1|^2 + \cdots + a_n|z_n|^2 
\] 
on $K^n$ where $a_1, \ldots, a_n\in k$. Let $G_{sc}$ be the group of special unitary transformations, and $T$ the torus 
consisting of diagonal matrices. Then $T(k)\cong (K^1)^n $ where $K^1$ is the group of norm one elements in $K^{\times}$. 
Let $\sigma\in \Gamma_k\setminus\Gamma_K$. Then $\sigma(z)=z^{-1}$ for every $z\in K^1$. Hence 
 ${}^{\sigma}\lambda=-\lambda$ for all weights. Since $-1$ is not 
in the root system of type $\rA_{n-1}$, $\sigma$ acts non-trivially on the Dynkin diagram.

\subsection{}  Let $K$ be the extension of $k$ such that the Galois group 
 $\Gamma_K$ is the stabilizer of all vertices of  the Dynkin diagram.  
For every  $\lambda \in \Lambda_+$  there exists a unique irreducible $\bar k$-representation 
$V_\lambda$ of $G_{sc}(\bar k)$ with the highest weight $\lambda$.    
Tits shows that $V_{\lambda}$ can be defined over $K$ in the following sense. 
There exists a central simple algebra $\calA_{\lambda}$ over $K$  of dimension $(\dim V_{\lambda})^2$ and 
a homomorphism 
\[ \rho_\lambda : G_{sc}(K)  \rightarrow \GL_1(\calA_{\lambda}) \] 
such that $V_{\lambda}$ is obtained from $\rho_{\lambda}$ by extension of scalars.  
Define $\beta_{G}(\lambda) \in \Br(K)$ to be the class of $\calA_{\lambda}$ in the Brauer group of $K$.  In
\cite{T2}, Tits shows that $\beta_G$ induces a group
homomorphism
\[
\beta_G : C^* \rightarrow \Br(K). 
\]
For example, if $G_{sc}$ is a form of $\SL_n$, corresponding to a degree $n$ central simple algebra $\calA$ over a quadratic extension $K$, with a unitary involution, then 
$\beta_G(\lambda)$  is the class of $\calA$, where $\lambda$ is the highest weight of an $n$-dimensional representation of $G_{sc}(\bar k)$. 

\subsection{} We now give a cohomological interpretation of the map
$\beta_G$.  Let $G=G_{sc}/C$ be the adjoint quotient of $G_{sc}$. We
have an exact sequence of algebraic groups
\[
1 \rightarrow C \rightarrow G_{sc} \rightarrow G \rightarrow 1.
\]
which gives a coboundary map $\delta : H^1(k, G) \rightarrow H^2(k,
C)$.  Let $A = \Aut(\frakg)$. The inclusion of~$G $ in $ A$ gives rise
to a map of pointed sets $j : H^1(k,G) \rightarrow H^1(k,A)$. As in
Section \ref{S2.1}, let $c_G \in H^1(k,G)$ be the unique point such
that $j(c_G)$ is a class of a quasi split Lie algebra.  Let
$t_G=-\delta(c_G)\in H^2(k,C)$.  Any weight $\lambda$ gives a
homomorphism $C\rightarrow \bar k^{\times}$. Hence we have a sequence
of maps
\[ 
H^2(k,C)\longrightarrow H^2(K, C)\longrightarrow H^2(K, \bar k^{\times})=\Br(K)
\]
where the first map is the restriction.  Let $h_{\lambda}:
H^2(k,C)\longrightarrow \Br(K)$ be the composite.  By Proposition 31.6
in \cite{BoI}
\begin{equation} \label{DF} 
\beta_G(\lambda)=h_{\lambda}(t_G).
\end{equation}

\subsection{} Let $(\frakh,\frakg)$ be a Tits pair. Assume that the type of $\frakg$ is not $\rF_4$.   In this case $C_H=C_G=C$, 
hence the commutative 
diagram (\ref{BIG})  gives a commutative diagram 
\[
\xymatrix{ H^1(k, H) \ar[r]^{\delta_H} \ar[d]_{\alpha} & 
H^2(k, C) \ar@{=}[d] \\ 
H^1(k, G) \ar[r]^{\delta_G} & H^2(k, C)}
\]
Let $\frakh^{qs}$ be the unique quasi split inner form of $\frakh$.
This gives a unique class $c_H\in H^1(k,H)$.  By Proposition
\ref{P35}, $\gamma(\frakh^{qs})$ is a class of a quasi split
$\frakg^{qs}$.  Hence $\alpha(c_H)=c_G$, and the formula~\eqref{DF}
implies the following proposition.

\begin{prop}\label{P:tits} Let $(\frakh,\frakg)$ be a Tits pair. Assume that the type of $\frakg$ is not $\rF_4$.  Then 
\[ 
\beta_H=\beta_G. \qed
\] 
\end{prop}

\section{Forms of type $\rE_6$}

\subsection{} Assume that $G_{sc}$ has type $\rE_6$.  Let $K$ be the smallest extensions of  $k$ such that the Galois group  $\Gamma_K$ acts trivially on the Dynkin diagram of $G_{sc}$.  The  center $C$ of  $G_{sc}$ is $\Gamma_K$-isomorphic to $\mu_3$. Hence the image of $\beta_G$ is contained 
in the 3-torsion of the Brauer group: 
\[ 
\beta_G: C^* \rightarrow \Br_3(K).  
\] 
 Let $\varpi$ and $\varpi'$ be the fundamental weights of the two 27 dimensional representations of~$G_{sc}(\bar k)$.  Since $\varpi + \varpi'$ is in the root lattice,  
 $\beta_G(\varpi)$ and $\beta_G(\varpi')$ are classes of two opposite division algebras. 
  If $K$ is a quadratic extension of $k$, then $\sigma\in \Gamma_k\setminus\Gamma_K$ 
  switches $\varpi$ and $\varpi'$.  This implies that $\beta_G(\varpi)$, with $K$-structure twisted by $\sigma$, is isomorphic to 
  its opposite $\beta_G(\varpi')$ i.e. $\beta_G(\varpi)$ admits an involution of the second kind with respect to $K$, see also \cite{T2}.

\begin{prop}  \label{P:outer} Assume that $(\frakh,\frakg)$ is a Tits pair over $k$ such that $\frakg$ has type $\rE_6$. 
\begin{enumerate}[(i)]
\item $\frakg$ has an outer automorphism defined over $k$ if and only if $\beta_G=1$.
\item $\frakg$ has an outer automorphism defined over $k$  if and only if $\frakh$ has. 
\item If $k=\bbR$ then both $\frakg$ and $\frakh$ have an outer automorphism defined over $\bbR$. 
\end{enumerate} 
\end{prop} 
\begin{proof} 
  If $G_{sc}(k)$ has an outer automorphism then
  $\beta_G(\varpi)=\beta_G(\varpi')$. Hence $\beta_G=1$, since
  $\beta_G(\varpi)$ and $\beta_G(\varpi')$ are inverses of each other,
  and of order dividing 3.  If $\beta_G=1$, then $\beta_H=1$ by
  Proposition \ref{P:tits}.  Since $\beta_H=1$, the factor of $\frakh$
  of type $\rA_2$ is defined by means of a unitary involution on the
  matrix algebra $M_3(K)$. But then $\frakg$ has an outer automorphism
  defined over $k$, as constructed in the proof of Lemma
  \ref{L29}. This proves (i).  Part (ii) now also follows, as
  existence of the outer automorphism for both algebras is equivalent
  to the vanishing of $\beta_G=\beta_H$.  Part (iii) follows from the
  fact that $\Br_3(\bbR)=\Br_3(\bbC)=1$, hence $\beta_G$ is trivial.
\end{proof}

\subsection{} \label{S51}
Assume that $(\frakh, \frakg)$ be a Tits pair such that  $\frakg$ is of type $\rE_6$.  
We have a commutative diagram
\[
\xymatrix{ 1 \ar[r] & H(\ok)  \ar[r] \ar[d] &
  B(\ok) \ar[r] \ar[d] & \bbZ/2\bbZ \ar[r] \ar@{=}[d] & 1 \\ 
1 \ar[r] & G(\ok) \ar[r] & A(\ok) \ar[r] &
     \bbZ/2\bbZ \ar[r] & 1}
\]
where the two horizontal rows are exact sequences. This in turn gives a commutative diagram
\begin{equation} \label{X}
\xymatrix{1  \ar[r] & H(k) \ar[r]\ar[d] & B(k)  \ar[r] \ar[d] & \bbZ/ 2 \bbZ \ar[r] \ar@{=}[d] & H^1(k,H )
  \ar[r]^{j_H} \ar[d]_{\alpha} & H^1(k,B ) \ar[r] \ar[d]_{\gamma} & H^2(k,  \bbZ/2\bbZ) \ar@{=}[d]  
  \\ 
1 \ar[r]  &  G(k) \ar[r] &  A(k) \ar[r] & \bbZ/ 2 \bbZ \ar[r] & H^1(k,G ) \ar[r]^{j_G} & H^1(k,A ) \ar[r] &  H^2(k,  \bbZ/2\bbZ)}
\end{equation}

\begin{lemma}\label{L:reduction}  Let $x\in H^1(k,G)$ and $y=j_G(x)\in H^1(k,A)$.  Then $j_H$ is a bijection from 
$\alpha^{-1}(x)$ to $\gamma^{-1}(y)$.   
\end{lemma} 
\begin{proof} Let $z\in \gamma^{-1}(y)$.  A diagram chase shows that $z=j_H(w)$ for some $w\in H^1(k,H)$. 
After twisting by $w$, we can assume that $y=[\frakg]$ and $z=[\frakh]$. 
Now the diagram (\ref{X}) and Proposition \ref{P:outer}, part (ii) imply that $\alpha$ gives a bijection between 
$j_{H}^{-1}(z)$ and $j_G^{-1}(y)$. 
 \end{proof} 

\begin{lemma}\label{L:reduction2}  The Tits construction is surjective (respectively injective) if and only if 
the map $\alpha$ is surjective (respectively injective) for all pairs $(\frakh,\frakg)$.   
\end{lemma} 
\begin{proof} Let $y\in H^1(k,A)$. We can pick a pair $(\frakh,\frakg)$ such that $y=j_G(x)$ for some $x\in H^1(k,G)$.  
For example, we can pick both algebras split or quasi-split. Lemma now follows from Lemma \ref{L:reduction}.
 \end{proof}

\section{Real Lie algebras}

In this section, all groups and Lie algebras are defined over $\bbR$
unless otherwise stated.

\subsection{Real forms of type $\rG_2$} 
The compact real Lie algebra
$\frakg_2^\rmc=\Der(\bbO^\rmc)$ and the split real Lie algebra $\frakg_2^\rms=\Der(\bbO^\rms)$ are
the only real Lie algebras of type $\rG_2$. 

\subsection{Real Jordan algebras}
Recall, from Section \ref{S26}, the three Jordan algebras
$\calJ^\rms$, $\calJ^\rmc$ and $\calJ^{\rmII}$ constructed from a
composition algebra of dimension $d$.  Then $\Der(\calJ^\rms)$ is
split, $\Der(\calJ^\rmc)$ is compact, and $\Der(\calJ^{\rmII})$ has
split rank one, as specified by the following table.

\[
\mbox{Table1} \ \ \
\begin{array}{|c||c|c|c|c|}
\hline d & 1 & 2 & 4 & 8 \\ 
\hline 
\Der(\calJ^\rmII) & \frakso_{1,2}(\bbR) & \fraksu_{1,2}(\bbC) & 
      \fraksu_{1,2}(\bbH) & \frakf_{4,1} \\
\hline
\end{array}
\]
In the literature, $\frakf_{4,1}$ is also denoted by $\frakf_{4(-20)}$.

\subsection{ Tits construction over real numbers} \label{S61} Let
$\frakg= \frakg(\bbO, \calJ)$ be a real exceptional Lie algebra
obtained by the Tits construction in Section~\ref{S28}. We shall
determine this real form by computing the signature of the Killing
form. This is the same as computing the trace of the corresponding
Cartan involution.  Let $\frakh =\Der(\bbO)\oplus \Der(\calJ)\subseteq
\frakg$.  Let $\theta_H$ be a Cartan involution of $\frakh$
corresponding to a maximal compact subalgebra of $\frakh$. Since the
maximal compact subalgebra of $\frakh$ is contained in a maximal
compact subalgebra of $\frakg$, the Cartan involution $\theta_H$
extends to a Cartan involution $\theta_G$ of $\frakg$ corresponding to
the maximal compact subalgebra of $\frakg$.  We claim that $\theta_G$
is given by the natural action of $\theta_H\in \Aut(\frakh) \subseteq
\Aut(\frakg)$ on $\frakg$.  Indeed, $\theta_H\theta_G^{-1}\in
\Aut(\frakg)$ centralizes $\frakh$. By Lemma \ref{L29}, the centralizer
of $\frakh$ in $\Aut(\frakg)$ is trivial and the claim follows. Using the
explicit action of $\theta_H$ on $\bbO^0\otimes \calJ^0$, it is not
difficult to compute the trace of $\theta_G$ on $\frakg$. We tabulate
the results (\cite{Ja}):


\[
\begin{array}{|c||c|c|c|c|c|c|c|c|c|c|c|}
  \hline \bbO & \calJ^\rmc(1) &  \calJ^\rms(1) & \calJ^\rmc(2) & \calJ^\rmII(2) &
  \calJ^\rms(2) & \calJ^\rmc(4) & \calJ^\rmII(4) & \calJ^\rms(4) &  \calJ^\rmc(8) &
  \calJ^\rmII(8) & \calJ^\rms(8) \\  
  \hline \bbO^\rmc  & \frakf_4^\rmc & \frakf_{4,1} & \frake_6^\rmc &
  \frake_{6(-14)} &  \frake_{6(-26)} & \frake_7^\rmc & \frake_{7,4} &
  \frake_{7,3} & \frake_8^\rmc & \frake_8^\rms & \frake_{8,4} \\ 
  \bbO^\rms & \frakf_4^\rms & \frakf_4^\rms & \frake_{6,4} &
  \frake_{6,4} & \frake_6^\rms & \frake_{7,4} & \frake_{7,4} &
  \frake_7^\rms & \frake_{8,4} & \frake_{8,4} & \frake_8^\rms
  \\ \hline
\end{array}
\]
\centerline{Table 2}

In the above table, $d$ in $\calJ(d)$ is the dimension of the
composition algebra used to construct~$\calJ$.  The upper-scripts s
and c denote the split and compact Lie algebras, respectively.  The
algebra $\frake_{6(-14)}$ (respectively $\frake_{6(-26)}$) is the
unique Lie algebra of type $\rE_6$ with real rank 2 whose Cartan
decomposition $\frakk \oplus \frakp$ satisfies $\dim \frakp - \dim
\frakk = -14$ (respectively $-26$).  By inspection, we see that Table
2 contains all real exceptional Lie algebras of type $\rF_4$ and
$\rE_n$.  Thus we have obtained the following (see \cite{Kv} for a
more thorough discussion of dual pairs in $\frakg$ of type~$\rE_n$).

\begin{prop} If $k$ is an archimedean  field, then $\gamma: H^1(k,B) \longrightarrow H^1(k,A)$ is a surjection.  
If $k\cong\bbC$ then $\gamma$ is a bijection. 
If $k\cong\bbR$ then  $\gamma$ is given by Table 2. 
\end{prop}

\subsection{Real forms of type $\rE_6$}\label{S6.5}
Recall that $H$ and $G$ are the connected components of $B=\Aut(\frakh)$ and $A=\Aut(\frakg)$, respectively. 
We have a commutative diagram
\[
\xymatrix{  H^1(\bbR,H )
  \ar[r]^{j_H} \ar[d]_{\alpha} & H^1(\bbR,B )  \ar[d]_{\gamma} & \\
    H^1(\bbR,G ) \ar[r]^{j_G} & H^1(\bbR,A )  }
\]

By Proposition \ref{P:outer}(iii)  the maps $j_H$ and $j_G$ are injections. 
 Thus, the elements of $H^1(\bbR,H)$ and $H^1(\bbR,G)$ can be identified with the inner forms of $\frakh$ and $\frakg$ 
respectively. Hence the map $\alpha$ is also given by Table 2. In particular, it is surjective.

\section{Rational forms over $p$-adic fields} \label{S4}

\subsection{} \label{S41} 
In this section $k$  is a $p$-adic field. There is a unique (split) octonion algebra over $k$, 
denoted by $\bbO$.  The algebra $\Der(\bbO)$ is a split simple Lie algebra of type $\rG_2$. 
Let $\calJ$ be a cubic Jordan algebra over $k$ and 
$\frakg=\frakg(\bbO,\calJ)$ be the corresponding exceptional Lie algebra obtained by the Tits construction. 
If $\calJ$ is split, then $\Der(\calJ)$ is also split. Other cases are tabulated below: 
\[
\mbox{Table 3}  \ \ \ 
\begin{array}{|c||c|c|c|c|}
\hline \calJ & \calH_3(k,\rho)  & D^+ & \calH_3(K) & \calH_3(Q) \\ 
\hline 
\Der(J) & \frakso_3 (k) & \fraksl_{1}(D) & \fraksu_{3}(K) & 
      \fraksu_{3}(Q) \\
\hline
\end{array}
\]
In the table $D$ is a cubic division algebra over $k$ and $\fraksl_1(D)$ is the Lie algebra
consisting of the trace zero elements in $D$; $K$ is a quadratic extension of $k$, and $Q$ is the unique
 quaternion algebra over $k$.  Over the field $k$ there are two cubic division algebras. If $D$ is one then 
 the other is its opposite $D^{\circ}$. 
 The algebras $\fraksl_{1}(D)$ and $\fraksl_{1}(D^{\circ})$ are isomorphic by the map $x\mapsto -x$ for all 
 $x\in \fraksl_1(D)$.

\subsection{}    By a result of Kneser, see \cite{Kn},  $H^1(k, H_{sc})$ and $H^1(k, G_{sc})$ are trivial. 
Hence the commutative diagram (\ref{BIG})  gives a commutative diagram 

\[
\xymatrix{ H^1(k, H) \ar[r]^{\delta_H} \ar[d]_{\alpha} & 
H^2(k, C_H) \ar[d] \\ 
H^1(k, G) \ar[r]^{\delta_G} & H^2(k, C_G)}
\]
where horizontal maps are bijections. The following proposition now
follows from Lemma~\ref{L:reduction}.

\begin{prop} \label{P73} If $k$ is a $p$-adic field, then $\gamma:
  H^1(k,B) \longrightarrow H^1(k,A)$ is a surjection. It is a
  bijection if $\frakg$ is not of type $\rF_4$.
\end{prop}

The group $H^2(k,C_H)$ can be computed by identifying $C_H$ with a
concrete group as in the following table. In the table, $\mu_3[K]$ is
the group of cube roots of 1, where the usual action of $\Gamma_k$ is
twisted by $\sigma(\zeta)= \zeta^{-1}$ for every $\sigma\in
\Gamma_{k}\setminus\Gamma_K$.  The data given in the last row follow
from results in local class theory, such as $H^2(k, \mu_n)=Br_n(k)=
\mathbb Z/n\mathbb Z$. The group $H^2(k,\mu_n[K] )$, for $n$ odd,
parameterizes division algebras over $K$ of degree dividing $n$, that
admit a unitary involution, see (30.15) in \cite{BoI}. On the other
hand, a division algebra of degree dividing $n$ over $K$ admits a
unitary involution if and only if it is in the kernel of the
corestriction map from $Br_n(K)$ to $Br_n(k)$. However this map is a
bijection for the $p$-adic field $k$ and $n$ odd, hence
$H^2(k,\mu_n[K] )$ is trivial.
\[
\mbox{Table 4}  \ \ \ 
\begin{array}{|c||c|c|c|c|c|}
\hline G &  \rF_4 & \rE_6 & \rE_6^K & \rE_7 & \rE_8 \\ 
\hline 
C_H& \mu_2 & \mu_3 & \mu_3[K]  &  \mu_2 & 1 \\
       \hline 
H^2(k,C_H)& \bbZ/2\bbZ   & \bbZ/3\bbZ & 1 & 
      \bbZ/2\bbZ & 1 \\
\hline
\end{array}
\]
Completeness of data in Table 3 follows from Table 4. We have a unique dual pair, except when $\frakg$ has type $\rF_4$. 
If $\frakg$ is of type $\rF_4$, then it is split and it contains two conjugacy classes of dual pairs, with split and non-split $\frakso_3(k)$.

\section{Rational forms over global fields} \label{S7}

\subsection{} We have a commutative diagram 
\[
\xymatrix{  
& H^1(k,H )
  \ar[r]^{i_H} \ar[d]_{\alpha} & \prod_v H^1(k_v,H ) \ar[r]^{ \prod_v \delta_{H_v}} \ar[d]_{\prod_v\alpha_v} &
    \prod_v H^2(k_v,C_H)  \ar[d] & \\ 
  & H^1(k,G ) \ar[r]^{i_G} &  \prod_v H^1(k_v,G) \ar[r]^{\prod_v \delta_{G_v}} &
   \prod_v H^2(k_v,C_G) & }
\]
where $\prod_v$ denotes  a restricted direct product of pointed sets over all places $v$ of $k$, i.e. if $(x_v)$ is an element of the product, then $x_v$ is the distinguished point for almost all $v$. The maps 
 $i_H$ and $i_G$ are injections (\cite{PR} Theorem 6.22) while $\prod_v\alpha_v$ is a surjection. 
 The image of $i_H$ is equal to the inverse image of $H^2(k,C_H)\subseteq  \prod_v  H^2(k_v,C_H)$ (\cite{Kz}, 2.6 Proposition). The same also applies to the image of $i_G$.
 Now a simple diagram chase shows that $\alpha$ is a surjection for every pair $(\frakh,\frakg)$.
  It follows, from Lemma \ref{L:reduction2}, that 
\[ 
\gamma: H^1(k,B)\longrightarrow  H^1(k,A)
\] 
is surjective. 
The question of fibers of $\gamma$ can be translated to the question of fibers for $\alpha$ by Lemma \ref{L:reduction}. 
If $\frakg$ is not of type $\rF_4$ then, by Proposition \ref{P73},  
$\alpha_v$ is a bijection for every finite place, hence the fibers of $\alpha$ can be identified with the fibers of 
$\alpha_{\infty}: H^1(k_{\infty}, H) \rightarrow H^1(k_{\infty},G)$. These are the same as the fibers of 
$\gamma_{\infty}: H^1(k_{\infty}, B) \rightarrow H^1(k_{\infty},A)$ by Section  \ref{S6.5}. Theorem~\ref{main} is proved.

\subsection{} We shall now show how to construct  given $\frakg$ using the Tits construction. Let 
 $\bbO$ be an octonion $k$-algebra. By the Hasse principle for  $\rG_2$ the octonion algebra $\bbO$ is specified by isomorphism classes 
of its completions $\bbO_v$ for all real places $v$.   

\begin{enumerate}[(a)]
\item{$\rF_4$} Let $\calA=\rM_3(k)$.  Let $\calJ=S(\sigma,\calA)$ where 
$\sigma$ is an orthogonal involution of $\calA$.  Pick $\bbO$ and the involution $\sigma$ such that 
 $\frakg_v\cong \frakg(\bbO_v,\calJ_v)$ for all real places $v$.  Then $\frakg\cong \frakg(\bbO,\calJ)$.

\item{$\rE_6$} (inner) Let $\varpi$ be the highest weight of an irreducible representation of dimension 27. Let $\calA$ be 
a central simple algebra of degree 3 such that $[\calA]=\beta_G(\varpi)\in \Br_3(k)$. Let $\calJ=\calA^+$.  
  Pick $\bbO$ such that 
 $\frakg_v\cong \frakg(\bbO_v,\calJ_v)$ for all real places $v$.  Then $\frakg\cong \frakg(\bbO,\calJ)$. 
 
\item{$\rE_6$} (outer) Let $\varpi$ be the highest weight of an irreducible representation of dimension 27. Let $\calA$ be 
a central simple algebra of degree 3 over $K$ such that $[\calA]=\beta_G(\varpi)\in \Br_3(K)$. Let $\calJ=S(\sigma,\calA)$ where 
$\sigma$ is a unitary involution of $\calA$.  Pick $\bbO$ and the involution $\sigma$ such that 
 $\frakg_v\cong \frakg(\bbO_v,\calJ_v)$ for all real places $v$.  Then $\frakg\cong \frakg(\bbO,\calJ)$.

\item{$\rE_7$} Let $\varpi$ be the highest weight of the irreducible representation of dimension 56. Let $\calA$ be 
a central simple algebra of degree 6 such that $[\calA]=\beta_G(\varpi)\in \Br_2(k)$. Let $\calJ=S(\sigma,\calA)$ where 
$\sigma$ is a symplectic involution of $\calA$.  Pick $\bbO$ and the involution $\sigma$ such that 
 $\frakg_v\cong \frakg(\bbO_v,\calJ_v)$ for all real places $v$.  Then $\frakg\cong \frakg(\bbO,\calJ)$. 
 
 \item{$\rE_8$}   By the Hasse principle for  $\rF_4$,  an exceptional Jordan  $k$-algebra $\calJ$ is specified by isomorphism classes 
of its completions $\calJ_v$ for all real places $v$. 
Pick $\bbO$ and the algebra $\calJ$ such that 
 $\frakg_v\cong \frakg(\bbO_v,\calJ_v)$ for all real places $v$.  Then $\frakg\cong \frakg(\bbO,\calJ)$. 

\end{enumerate} 

We refer the reader to the book by Scharlau \cite{Sc}, where existence of involutions  with prescribed local desiderata  is discussed.

\end{document}